\documentclass[a4paper,oneside,reqno]{amsart}

\usepackage[top=30mm,bottom=30mm,left=30mm,right=30mm]{geometry}
\usepackage{amsmath,mathtools,amssymb,amsthm}
\usepackage[foot]{amsaddr}
\usepackage[shortlabels]{enumitem}
\usepackage[all]{xy}

\usepackage[
 bookmarksnumbered=true,
 colorlinks=true,
 allcolors=blue
]{hyperref}

\usepackage[
 backend=biber,
 style=alphabetic,
 sorting=nyvt
]{biblatex}

\newcommand{\etale}{\mathrm{\acute{e}t}}
\newcommand{\abelian}{\mathrm{ab}}

\DeclareMathOperator{\Spec}{Spec}
\DeclareMathOperator{\Proj}{Proj}

\DeclareMathOperator{\PGL}{PGL}
\DeclareMathOperator{\Aut}{Aut}
\DeclareMathOperator{\SL}{SL}
\DeclareMathOperator{\PSL}{PSL}
\DeclareMathOperator{\Frac}{Frac}

\newcommand{\AbsGalGrp}[1]{G_{#1}}
\newcommand{\Abelianization}[1]{#1^{\abelian}}
\newcommand{\EtFundGrpWithPt}[2]{\pi_{1}^{\etale}(#1,#2)}
\newcommand{\EtFundGrp}[1]{\pi_{1}^{\etale}(#1)}
\newcommand{\FieldAlgeClosure}[1]{\overline{#1}}
\newcommand{\CardinalityOfSet}[1]{\lvert #1\rvert}

\numberwithin{equation}{section}

\theoremstyle{plain}
\newtheorem{theorem}{Theorem}[section]
\newtheorem*{theorem*}{Theorem}
\newtheorem{Itheorem}{Theorem}

\newtheorem{proposition}[theorem]{Proposition}
\newtheorem{lemma}[theorem]{Lemma}
\newtheorem{observation}[theorem]{Observation}
\newtheorem{corollary}[theorem]{Corollary}
\newtheorem{Iconjecture}[Itheorem]{Conjecture}

\theoremstyle{definition}
\newtheorem{definition}[theorem]{Definition}
\newtheorem{example}[theorem]{Example}
\newtheorem{remark}[theorem]{Remark}

\title[On branched coverings of the projective line over the integers]{On branched coverings of the projective line over the integers}
\date{Version of \today}

\author[R.~Shimizu]{Ryoji Shimizu}
\address{Institute of Science Tokyo, 2-12-1 Ookayama, Meguro-ku, Tokyo 152-8550, Japan}
\email{shimizu.r.ap@m.titech.ac.jp}

\author[N.~Yamaguchi]{Naganori Yamaguchi}
\address{Tokyo Denki University, 5 Senju-Asahi-cho, Adachi-ku, Tokyo 120-8551, Japan}
\email{n.yamaguchi@mail.dendai.ac.jp}

\thanks{This work was supported by JSPS KAKENHI Grant Numbers 25KJ0125 and 23KJ0881.}

\begin{document}

\begin{abstract}
	We investigate the \'etale fundamental group of the complement of a horizontal divisor on $\mathbb{P}^{1}_{\mathbb{Z}}$.
	We prove that this group has no nontrivial finite solvable quotient if and only if the divisor is normal crossings at the prime~$2$.
	Moreover, if the divisor is normal crossings at the prime~$2$ and either has three irreducible components or is normal crossings at the prime~$3$, we show that no quotient isomorphic to $\PSL_{2}(q)$ can occur for certain prime powers~$q$.
\end{abstract}

\maketitle
\tableofcontents

\section*{Introduction}\label{intro}

For each $a\in \mathbb{P}^{1}_{\mathbb{Q}}(\mathbb{Q})$, let $D_a\subset \mathbb{P}^{1}_{\mathbb{Z}}$ denote the Zariski closure of $a$.
That is, if $a=\frac{u}{v}\in \mathbb{Q}$ is written in lowest terms with $v\in\mathbb{Z}_{>0}$, then $D_a$ is defined on the affine chart with coordinate $t$ by the equation $vt-u=0$.
We write $D_{\infty}$ for the divisor at infinity.
Let $D \coloneq \sum_{i=1}^{r} D_{a_i}$ be a reduced effective Cartier divisor on $\mathbb{P}^{1}_{\mathbb{Z}}$, where the $a_i$ are distinct.
We say that $\mathbb{P}^{1}_{\mathbb{Z}} \setminus D$ has \emph{trivial \'{e}tale fundamental group} if
\begin{equation*}
	\EtFundGrp{\mathbb{P}^{1}_{\mathbb{Z}} \setminus D}=1.
\end{equation*}
In~\cite[Appendix]{MR1305400}, Y.~Ihara recalled the following theorem of T.~Saito, which gives a sufficient geometric condition for this property:

\bigskip

\begin{theorem*}
	If $D$ is a normal crossings divisor, then $\mathbb{P}^{1}_{\mathbb{Z}} \setminus D$ has trivial \'{e}tale fundamental group.
\end{theorem*}

\bigskip

\noindent
Motivated by Saito's theorem, we ask whether its converse holds.
Unfortunately, this is false.
Indeed, if $p$ is an odd prime, then the divisor $D_{0}+D_{p^{2}}+D_{\infty}$ is not normal crossings at the prime~$p$.
However, the scheme $\mathbb{P}^{1}_{\mathbb{Z}} \setminus (D_{0}+D_{p^{2}}+D_{\infty})$ has trivial \'etale fundamental group (see Example~\ref{ex:p-square}).

\medskip

On the other hand, the preceding construction does not give a counterexample when $p=2$.
In fact, if $D$ is not normal crossings at the prime~$2$, we can construct a surjection $\EtFundGrp{\mathbb{P}^{1}_{\mathbb{Z}} \setminus D}\twoheadrightarrow \mathbb{Z}/2\mathbb{Z}$ (see Proposition~\ref{prop:only-if}).
In particular, we have the following:

\medskip

\begin{quote}\it
	If $\mathbb{P}^{1}_{\mathbb{Z}} \setminus D$ has trivial \'etale fundamental group, then $D$ is normal crossings \emph{at the prime~$2$}.
\end{quote}

\medskip

\noindent
These results suggest the following conjecture:

\begin{Iconjecture}\label{Iconjecture_NCat2}
	If $D$ is normal crossings at the prime~$2$, then $\mathbb{P}^{1}_{\mathbb{Z}} \setminus D$ has trivial \'etale fundamental group.
\end{Iconjecture}

\noindent
If the conjecture holds, then $\mathbb{P}^{1}_{\mathbb{Z}} \setminus D$ has trivial \'etale fundamental group if and only if $D$ is normal crossings at the prime~$2$.
Similar questions make sense over number fields and over rings of $S$-integers.
In the present paper, we restrict attention to $\mathbb{P}^{1}_{\mathbb{Z}}$ and prove several results toward Conjecture~\ref{Iconjecture_NCat2}.
Our first result is obtained by passing to the maximal prosolvable quotient.
For a profinite group $G$, write $G^{\mathrm{solv}}$ for its maximal prosolvable quotient.
The following theorem shows that Conjecture~\ref{Iconjecture_NCat2} holds at the solvable level:

\bigskip

\begin{Itheorem}[{Theorem~\ref{thm:solvable-quotient}}]\label{IntrothmB}
	If $D$ is normal crossings at the prime~$2$, then $\EtFundGrp{\mathbb{P}^{1}_{\mathbb{Z}} \setminus D}$ has no nontrivial finite solvable quotient.
	In particular,
	\begin{equation*}
		\EtFundGrp{\mathbb{P}^{1}_{\mathbb{Z}} \setminus D}^{\mathrm{solv}}=1,
	\end{equation*}
	if and only if $D$ is normal crossings at the prime~$2$.
\end{Itheorem}

\bigskip

\noindent
As an immediate consequence, Conjecture~\ref{Iconjecture_NCat2} holds if $r\notin\{3,4,5,6\}$.
Indeed, the case $r=2$ is controlled by the solvable quotient, and the case $r\geq 7$ is excluded by the observation that a divisor that is normal crossings at the prime~$2$ must satisfy $r\leq 6$ (see Corollary~\ref{cor:easy-range}).

\medskip

If $D$ is normal crossings at the prime~$2$ and $\EtFundGrp{\mathbb{P}^{1}_{\mathbb{Z}} \setminus D}$ is nontrivial, then it has a nontrivial finite quotient of minimal order. By Theorem~\ref{IntrothmB}, such a quotient is non-solvable, and by minimality it is a non-abelian finite simple group.
Thus, after proving Theorem~\ref{IntrothmB}, the remaining part of Conjecture~\ref{Iconjecture_NCat2} is reduced to excluding non-abelian finite simple quotients.
We first record the following uniform obstruction in terms of the primes at which $D$ is normal crossings (see Theorem~\ref{thm:general-nonsolvable-obstruction}):

\medskip

\begin{quote}\it
	Let $G$ be a nontrivial finite group, and write $o(G)$ for the set of prime divisors of $\CardinalityOfSet{G}$.
	Assume that $D$ is normal crossings at every prime in $o(G)$.
	Then there is no surjective homomorphism
	\begin{equation*}
		\EtFundGrp{\mathbb{P}^{1}_{\mathbb{Z}} \setminus D}\twoheadrightarrow G.
	\end{equation*}
\end{quote}

\medskip

\noindent
A useful consequence of the Feit--Thompson odd-order theorem is that every finite non-solvable group has even order; in particular, $2\in o(G)$ for every such group $G$.
Let us refine this observation for several small simple groups.
An important fact from finite group theory is that every finite non-solvable group has a non-abelian finite simple composition factor.
The classification of finite simple groups is as follows:
\begin{itemize}
	\item cyclic groups of prime order;
	\item alternating groups $A_{n}$ for $n\geq 5$;
	\item finite simple groups of Lie type;
	\item the $26$ sporadic simple groups.
\end{itemize}
From this point of view, we begin with the smallest non-abelian simple groups.
In the present paper, we focus on the family of projective special linear groups $\PSL_{2}(q)$ over $\mathbb{F}_{q}$.
(Note that they are simple if and only if $q\geq 4$; see, for instance,~\cite[Theorem~24.17 and Example~24.22]{MR2850737}.)
In particular, since there are isomorphisms
\begin{equation*}
	\PSL_{2}(4)\cong A_{5},\qquad \PSL_{2}(9)\cong A_{6},
\end{equation*}
our results include the cases of $A_{5}$ and $A_{6}$.

\bigskip

\begin{Itheorem}[{Theorem~\ref{thm:small-simple-obstruction}}]\label{IthmD}
	Assume that $D$ is normal crossings at the prime~$2$.
	Assume moreover that either $r=3$ or $D$ is normal crossings at the prime~$3$.
	Let $q\geq 4$ satisfy one of the following conditions:
	\begin{enumerate}[(a)]
		\item $q=2^{f}$ for some $f\geq 2$;
		\item $q=3^{f}$ for some $f\geq 2$;
		\item $q$ is an odd prime power such that $2$ is not a square in $\mathbb{F}_{q}$.
	\end{enumerate}
	Then there is no surjective homomorphism
	\begin{equation*}
		\EtFundGrp{\mathbb{P}^{1}_{\mathbb{Z}} \setminus D}\twoheadrightarrow \PSL_{2}(q).
	\end{equation*}
\end{Itheorem}

\bigskip

The paper is organized as follows.
In Section~\ref{sec:abhyankar}, we review Abhyankar's lemma.
In Section~\ref{sec:triviality}, we study divisors on $\mathbb{P}^{1}_{\mathbb{Z}}$ and establish Theorem~\ref{IntrothmB}.
In Section~\ref{sec:simple-quotients}, we study finite non-solvable quotients and prove Theorem~\ref{IthmD}.

\section{\'etale fundamental groups and Abhyankar's lemma}
\label{sec:abhyankar}

We first recall the \'etale fundamental group and its boundary inertia groups.
Let $X$ be a connected, locally Noetherian, and normal scheme.
Let $D$ be a reduced effective Cartier divisor on $X$, and set $U \coloneq X \setminus D$.
By~\cite{MR0354651}, the category of \'etale coverings of $U$ is a Galois category.
Hence, for a geometric point $\bar\eta \to U$, we denote the associated profinite group by
\begin{equation*}
	\EtFundGrpWithPt{U}{\bar\eta}
\end{equation*}
and call it the \emph{\'etale fundamental group} of $U$.
Since this group is independent of the base point up to inner automorphism, we fix a geometric point $\bar\eta \to U$ and often omit $\bar\eta$ from the notation.
Let $D_i$ be an irreducible component of $D$, and let $\eta_i$ be its generic point.
We write $\mathcal{O}_{X,\eta_i}^{\mathrm{sh}}$ for the strict henselization of $\mathcal{O}_{X,\eta_i}$ with respect to a geometric point over $\eta_i$, and set $K_i^{\mathrm{sh}} \coloneq \Frac(\mathcal{O}_{X,\eta_i}^{\mathrm{sh}})$.
Then
\begin{equation*}
	\Spec(K_i^{\mathrm{sh}})
	=
	\Spec(\mathcal{O}_{X,\eta_i}^{\mathrm{sh}})
	\setminus \{\text{closed point}\}.
\end{equation*}
Choose a geometric point $\bar\eta_i \to \Spec(K_i^{\mathrm{sh}})$ and a geometric path from its image in $U$ to $\bar\eta$.
The induced morphism $\Spec(K_i^{\mathrm{sh}}) \to U$ gives a continuous homomorphism
\begin{equation*}
	\EtFundGrpWithPt{\Spec(K_i^{\mathrm{sh}})}{\bar\eta_i}
	\to
	\EtFundGrp{U}.
\end{equation*}
We call the image of this homomorphism the \emph{inertia group $I_{D_i}$ of $\EtFundGrp{U}$ attached to $D_i$}.
This subgroup is well defined up to conjugacy.
If $G$ is a finite quotient of $\EtFundGrp{U}$, we denote by $I_{D_i}(G)\subset G$ the image of an inertia subgroup attached to $D_i$.
\medskip

Next, we recall Abhyankar's lemma and some of its consequences, which are key ingredients in the later sections.
We begin by recalling the notions of normal crossings and codimension-$1$ tame ramification.
For related approaches to tame ramification on arithmetic schemes and comparisons among different notions of tameness, see Schmidt~\cite{MR1883386} and Kerz--Schmidt~\cite{MR2578565}.

\begin{definition}[{\cite[Definition~1.8.2]{MR316453}}]
	\label{def:snc}
	Let $X$ be a locally Noetherian scheme.
	An effective Cartier divisor $D$ on $X$ is called a \emph{strict normal crossings divisor} if, for every point $x \in D$, the local ring $\mathcal{O}_{X,x}$ is regular and there exist a regular system of parameters $t_1, \dots, t_d \in \mathfrak{m}_x$ and an integer $1 \leq r \leq d$ such that the ideal of $D$ in $\mathcal{O}_{X,x}$ is generated by $t_1 \cdots t_r$.
	Moreover, an effective Cartier divisor $D \subset X$ is called a \emph{normal crossings divisor} if, for every point $x \in D$, there exists an \'{e}tale morphism $U \to X$ whose image contains $x$ such that $D \times_X U$ is a strict normal crossings divisor on $U$.
\end{definition}

\begin{definition}[{\cite[Definition~2.2.2]{MR316453}}]
	\label{def:tame-codim1}
	Let $X$ be a connected, locally Noetherian, and normal scheme.
	Let $D$ be a normal crossings divisor on $X$, and set $U \coloneq X \setminus D$.
	A \emph{tame covering of $(X,D)$} is a finite morphism $f \colon Y \to X$ such that
	\begin{enumerate}[(i)]
		\item $Y$ is normal;
		\item the restriction $Y \times_X U \to U$ is \'{e}tale;
		\item for every irreducible component $D_i$ of $D$, the morphism $f$ is tamely ramified with respect to $\mathcal{O}_{X,\eta_{D_i}}$.
	\end{enumerate}
\end{definition}

\begin{theorem}[Absolute Abhyankar's lemma {\cite[APPENDICE I, Proposition~5.2]{MR0354651}}]
	\label{thm:abhyankar-local}
	We keep the notation of Definition~\ref{def:tame-codim1}.
	Let $f \colon Y \to X$ be a tame covering of $(X,D)$.
	Then, for every point $x \in X$, there exists an \'{e}tale neighbourhood $X' \to X$ of $x$ such that the pullback divisor $D' \coloneq D \times_X X'$ can be written as
	\begin{equation*}
		D' = D_1' + \cdots + D_r',
	\end{equation*}
	where $D_1', \dots, D_r'$ are irreducible regular divisors with strict normal crossings on $X'$, and such that $Y \times_X X'$ is a finite disjoint union of generalized Kummer coverings relative to $D_1', \dots, D_r'$ (see~\cite[Definition~1.3.8]{MR316453}).
\end{theorem}

\begin{corollary}
	\label{cor:corrected-prop-1-3}
	We keep the notation of Definition~\ref{def:tame-codim1}.
	Let $f \colon Y \to X$ be a tame covering of $(X,D)$.
	Fix an irreducible component $D_i$ of $D$ with generic point $\eta_i$.
	Set $K_i \coloneq \Frac(\mathcal{O}_{X,\eta_i})$.
	Let $L/K_i$ be the function field of an irreducible component of $Y \times_X \Spec(K_i)$, and let $B$ be the integral closure of $\mathcal{O}_{X,\eta_i}$ in $L$.
	Then, for every maximal ideal $\mathfrak{n} \subset B$ and every point $y \in D_i$, the ramification index $e(\mathfrak{n}/\eta_i)$ is invertible in $\mathcal{O}_{D_i,y}$.
\end{corollary}

\begin{proof}
	This follows directly from Theorem~\ref{thm:abhyankar-local}.
\end{proof}

\begin{corollary}
	\label{cor:tame-inertia-pro-complement}
	We keep the notation of Definition~\ref{def:tame-codim1}.
	Let $\mathfrak{Primes}$ denote the set of all primes, and set
	\begin{equation*}
		\mathbb{L}(D_i)
		\coloneq \mathfrak{Primes}\setminus
		\{\operatorname{char}(\kappa(s)) \mid s\in D_i \text{ is a closed point}\}.
	\end{equation*}
	Assume that $X$ is Nagata and that $D$ is horizontal over $\Spec(\mathbb{Z})$, i.e., every irreducible component of $D$ dominates $\Spec(\mathbb{Z})$.
	Then the inertia subgroup $I_{D_i}$ is a pro-$\mathbb{L}(D_i)$ group for every $i$.
\end{corollary}

\begin{proof}
    We may assume that $D$ is non-empty.
	Under the assumptions, the normalization $Y\to X$ of $X$ in any finite \'etale covering $V\to U$ is a tame covering of $(X,D)$.
	Indeed, since $X$ is Nagata, the normalization of $X$ in such a covering is finite over $X$ by~\cite[Tag~035S]{StacksProjectBook}; moreover, at the generic point of each irreducible component of $D$, the residue characteristic is $0$, hence every finite separable extension is tamely ramified.

	Let $N \unlhd \EtFundGrp{U}$ be an open normal subgroup with quotient $G \coloneq \EtFundGrp{U}/N$.
	Let $V\to U$ be the connected Galois covering corresponding to $G$, and let $Y\to X$ be the normalization of $X$ in $V$.
	By definition of $I_{D_i}$, its image $I_{D_i}(G)\subset G$ is the inertia subgroup of this covering along the discrete valuation ring $\mathcal{O}_{X,\eta_i}$.
	In particular, its order is equal to the ramification index $e(\mathfrak{n}/\eta_i)$ for a prime $\mathfrak{n}$ of the normalization of $\mathcal{O}_{X,\eta_i}$ in a field factor of $Y \times_X \Spec(K_i)$.

	Let $p$ be a prime such that $p\notin \mathbb{L}(D_i)$.
	By definition of $\mathbb{L}(D_i)$, there exists a closed point $s\in D_i$ with $\operatorname{char}(\kappa(s))=p$.
	By Corollary~\ref{cor:corrected-prop-1-3}, the ramification index $e(\mathfrak{n}/\eta_i)$ is invertible in $\mathcal{O}_{D_i,s}$.
	Hence $p\nmid e(\mathfrak{n}/\eta_i)$, and therefore $p$ does not divide $\CardinalityOfSet{I_{D_i}(G)}$.
	Since $p\notin\mathbb{L}(D_i)$ and $N$ are arbitrary, we conclude that $I_{D_i}$ is a pro-$\mathbb{L}(D_i)$ group.
\end{proof}

\section{Triviality and the normal crossings condition}\label{sec:triviality}

In this section, we discuss Conjecture~\ref{Iconjecture_NCat2}.
For each $a\in \mathbb{P}^{1}_{\mathbb{Q}}(\mathbb{Q})$, let $D_a\subset \mathbb{P}^{1}_{\mathbb{Z}}$ denote the Zariski closure of $a$.
That is, if $a=\frac{a_u}{a_d}\in \mathbb{Q}$ is written in lowest terms with $a_d\in\mathbb{Z}_{>0}$, then $D_a$ is defined by
\begin{equation*}
	V(a_dT_0-a_uT_1) \subset \Proj(\mathbb{Z}[T_0,T_1])=\mathbb{P}^{1}_{\mathbb{Z}}.
\end{equation*}
We write $D_{\infty}$ for the divisor at infinity, that is, $D_{\infty}=V(T_1)$.
Let $A$ be a finite subset of $\mathbb{P}^{1}_{\mathbb{Q}}(\mathbb{Q})$.
We write
\begin{equation*}
	D_A \coloneq \sum_{a\in A}D_a,\qquad\text{and}\qquad U_A \coloneq \mathbb{P}^{1}_{\mathbb{Z}} \setminus D_A.
\end{equation*}
If $G$ is a finite quotient of $\EtFundGrp{U_A}$, write $I_a(G)\coloneq I_{D_a}(G)$ for $a\in A$.
For a prime $p$, we say that $D_A$ is \emph{normal crossings at~$p$} if the pullback of $D_A$ to $\mathbb{P}^{1}_{\mathbb{Z}_{p}}$ is a normal crossings divisor on $\mathbb{P}^{1}_{\mathbb{Z}_{p}}$.

\begin{observation}
	\label{lem:geometric-surjectivity}
	Since $U_A\to\Spec(\mathbb{Z})$ is smooth and surjective with geometrically connected geometric generic fibre, the homotopy exact sequence~\cite[Preliminaries, Lemma~2]{MR659153}, together with
	$\EtFundGrp{\Spec(\mathbb{Z})}=1$, shows that the natural homomorphism
	\begin{equation*}
		\EtFundGrp{U_{A,\FieldAlgeClosure{\mathbb{Q}}}}
		\longrightarrow
		\EtFundGrp{U_A}
	\end{equation*}
	is surjective.
	Set $A=\{a_1,\dots,a_r\}$ with $r\geq 1$.
	After choosing topological generators of the boundary inertia groups, $\EtFundGrp{U_{A,\FieldAlgeClosure{\mathbb{Q}}}}$ is generated by the images of boundary inertia elements $\gamma_1,\dots,\gamma_r$ attached to $a_1,\dots,a_r$, and these images satisfy $\gamma_1\cdots\gamma_r=1$.
\end{observation}

\begin{lemma}
	\label{lem:good-prime-inertia-UA}
	Let $G$ be a finite quotient of $\EtFundGrp{U_A}$, and let $a\in A$.
	Let $p$ be a prime such that $D_A$ is normal crossings at~$p$.
	Then the inertia subgroup $I_a(G)$ is cyclic of order prime to $p$.
\end{lemma}

\begin{proof}
	The strict henselization of the local ring at the generic point of $D_a$ is a discrete valuation ring of residue characteristic~$0$.
	Its punctured spectrum has procyclic fundamental group (see~\cite[Expos\'{e} XIII, Corollaire~5.3]{MR0354651}), and hence $I_a(G)$ is cyclic.

	Let $Z\subset D_A$ be the set of all points where $D_A$ is not a normal crossings divisor, and set $X\coloneq\mathbb{P}^{1}_{\mathbb{Z}}\setminus Z$.
	Then $D_A|_X$ is a horizontal normal crossings divisor on the Nagata scheme $X$, and $X\setminus (D_A|_X)=U_A$.
	Since $D_A$ is normal crossings at~$p$, the component $D_a\cap X$ contains the point of $D_a$ above~$p$.
	Removing $Z$ changes neither $U_A$ nor the generic point of $D_a$, and hence the corresponding inertia subgroup is unchanged.
	Corollary~\ref{cor:tame-inertia-pro-complement} therefore shows that $\CardinalityOfSet{I_a(G)}$ is prime to~$p$.
\end{proof}

For a finite rational point $a\in\mathbb{P}^{1}_{\mathbb{Q}}(\mathbb{Q})$, write $a=a_u/a_d$ in lowest terms with $a_d\in\mathbb{Z}_{>0}$.
For distinct $a,b\in A$, define
\begin{equation*}
	\Delta(a,b)\coloneq
	\begin{cases}
		a_u b_d-a_d b_u, & a\neq \infty,\ b\neq \infty, \\
		a_d,             & b=\infty,                    \\
		b_d,             & a=\infty
	\end{cases}.
\end{equation*}
Note that for distinct $a,b\in \mathbb{P}^{1}_{\mathbb{Q}}(\mathbb{Q})$,  the number $v_p(\Delta(a,b))$ is invariant under the action of $\Aut(\mathbb{P}^{1}_{\mathbb{Z}})=\PGL_{2}(\mathbb{Z})$ on $\mathbb{P}^{1}_{\mathbb{Z}}$ for every prime $p$.

\begin{lemma}
	\label{lem:nc-at-prime-criterion}
	Let $p$ be a prime.
	Then, for any distinct $a,b\in A$, the divisors $D_a$ and $D_b$ meet in the special fibre $\mathbb{P}^{1}_{\mathbb{F}_{p}}$ if and only if $p\mid\Delta(a,b)$.
	When they meet in this fibre, their intersection consists of a unique closed point, and the local intersection multiplicity at that point is $v_p(\Delta(a,b))$.
	In particular, $D_A$ is normal crossings at~$p$ if and only if both of the following conditions hold:
	\begin{enumerate}[(i)]
		\item no closed point of $\mathbb{P}^{1}_{\mathbb{F}_{p}}$ lies on more than two components of $D_A$;
		\item whenever two distinct components $D_a$ and $D_b$ meet in $\mathbb{P}^{1}_{\mathbb{F}_{p}}$, we have $v_p(\Delta(a,b))=1$.
	\end{enumerate}
\end{lemma}

\begin{proof}
	We first consider the case $b=\infty$ and $a\neq\infty$.
	Write $a=a_u/a_d$.
	On the affine chart $T_0\neq 0$ with coordinate $w=T_1/T_0$, the divisor $D_\infty$ is given by $w=0$, while $D_a$ is given by $a_d-a_u w=0$.
	If $p\nmid a_d$, then $D_a$ and $D_\infty$ do not meet in the special fibre over $p$.
	Assume that $p\mid a_d$, and let $x$ be the unique closed point of $D_a\cap D_\infty\cap \mathbb{P}^{1}_{\mathbb{F}_{p}}$ corresponding to the maximal ideal $(p,w)$.
	We have
	\begin{equation*}
		\begin{aligned}
			i_x(D_a,D_\infty;\mathbb{P}^{1}_{\mathbb{Z}})
			 & =
			\mathrm{length}_{\mathbb{Z}[w]_{(p,w)}}\left(\frac{\mathbb{Z}[w]_{(p,w)}}{(w,a_d-a_u w)}\right) \\
			 & =
			\mathrm{length}_{\mathbb{Z}_{(p)}}\left(\frac{\mathbb{Z}_{(p)}}{(a_d)}\right)
			=
			v_p(a_d)
			=
			v_p(\Delta(a,\infty)).
		\end{aligned}
	\end{equation*}
	Hence the assertion follows in this case.

	Next, we consider the case $b\neq\infty$.
	Choose integers $r,s$ such that $r b_u+s b_d=1$.
	The automorphism
	\begin{equation*}
		M\colon(T_0:T_1)\longmapsto
		(rT_0+sT_1:b_dT_0-b_uT_1)
	\end{equation*}
	is defined by a matrix of determinant $-1$ and sends $D_b$ to $D_\infty$.
	Since $M$ is defined over $\mathbb{Z}$, it preserves the special fibres and local intersection multiplicities.
	Moreover, $Ma\neq\infty$ because $a\neq b$.
	The assertion now follows from the first case and the equality $v_p(\Delta(Ma,Mb))=v_p(\Delta(a,b))$.
	This completes the proof of the first assertion.

	After base change to $\mathbb{Z}_p$, each $D_a$ is a regular section of the regular surface $\mathbb{P}^{1}_{\mathbb{Z}_p}$.
	Hence the pullback of $D_A$ is normal crossings if and only if at most two components pass through each closed point and any two components that meet have intersection multiplicity~$1$.
	The final assertion therefore follows from the first assertion.
\end{proof}

\begin{lemma}
	\label{lem:triple-collision-gives-tangent-pair}
	Assume that three distinct irreducible components of $D_A$ pass through the same closed point of $\mathbb{P}^{1}_{\mathbb{F}_{2}}$.
	Then there exist distinct $a,b\in A$ such that $v_2(\Delta(a,b))\geq 2$.
\end{lemma}

\begin{proof}
	After applying an automorphism of $\mathbb{P}^{1}_{\mathbb{Z}}$, we may assume that the common point is the point $t=0$ on the affine chart $T_1\neq 0$.
	Every component passing through this point is of the form $D_{u_i/v_i}$ with $v_i$ odd and $u_i$ even.
	The classes $u_i/v_i$ modulo $4$ can take only the values $0$ and $2$.
	Hence, among three such components, two have the same class modulo $4$.
	For these two components, we have $u_i v_j-u_j v_i\equiv 0 \pmod{4}$, and therefore $v_2(\Delta(u_i/v_i,u_j/v_j))\geq 2$.
\end{proof}

\begin{proposition}
	\label{prop:only-if}
	Assume that $D_A$ is not normal crossings at the prime~$2$.
	Then there exists a continuous surjective homomorphism
	\begin{equation*}
		\EtFundGrp{U_A}\twoheadrightarrow
		\mathbb{Z}/2\mathbb{Z}.
	\end{equation*}
\end{proposition}

\begin{proof}
	By Lemmas~\ref{lem:nc-at-prime-criterion} and~\ref{lem:triple-collision-gives-tangent-pair}, there exist distinct $a,b\in A$ such that $v_2(\Delta(a,b))\geq 2$.
	After replacing $A$ by its image under an automorphism of $\mathbb{P}^{1}_{\mathbb{Z}}$, we may assume that $b=\infty$.
	Write $a=u/v$ in lowest terms with $v\in\mathbb{Z}_{>0}$.
	It follows that $4\mid v$ and that $u$ is odd.
	We have
	\begin{equation*}
		U_{\{a,\infty\}}
		=\Spec\left(\mathbb{Z}\left[t,\frac{1}{vt-u}\right]\right).
	\end{equation*}
	Since $U_A$ is a nonempty open subscheme of the regular integral scheme $U_{\{a,\infty\}}$, the natural homomorphism $\EtFundGrp{U_A}\twoheadrightarrow\EtFundGrp{U_{\{a,\infty\}}}$ is surjective.
	Choose $\varepsilon\in\{\pm1\}$ such that $\varepsilon u\equiv-1\pmod{4}$.
	Since $4\mid v$, we have
	\begin{equation*}
		f(x)\coloneq
		x^2+x+\frac{1-\varepsilon (vt-u)}{4}
		\in\mathbb{Z}[t][x].
	\end{equation*}
	The discriminant of $f$ is $\varepsilon (vt-u)$, which is a unit on $U_{\{a,\infty\}}$.
	Hence $f$ defines a finite \'{e}tale cover $Y\to U_{\{a,\infty\}}$ of degree~$2$.
	The element $\varepsilon (vt-u)$ is not a square in $\mathbb{Q}(t)$, since it has a simple zero.
	Thus $f$ is irreducible over $\mathbb{Q}(t)$.
	Since $f$ is monic, the coordinate ring of $Y$ embeds into $\mathbb{Q}(t)[x]/(f)$, and hence $Y$ is integral.
	Therefore $Y\to U_{\{a,\infty\}}$ is a connected finite \'{e}tale cover of degree~$2$.
	This proves the assertion.
\end{proof}

\begin{theorem}
	\label{thm:solvable-quotient}
	If $D_A$ is normal crossings at the prime~$2$, then $\EtFundGrp{U_A}$ has no nontrivial finite solvable quotient.
	In particular,
	\begin{equation*}
		\EtFundGrp{U_A}^{\mathrm{solv}}=1,
	\end{equation*}
	if and only if $D_A$ is normal crossings at the prime~$2$.
\end{theorem}

\begin{proof}
	Assume that $D_A$ is normal crossings at the prime~$2$.
	If $\CardinalityOfSet{A}\leq 1$, then $U_{A,\FieldAlgeClosure{\mathbb{Q}}}$ is isomorphic to either $\mathbb{P}^{1}_{\FieldAlgeClosure{\mathbb{Q}}}$ or $\mathbb{A}^{1}_{\FieldAlgeClosure{\mathbb{Q}}}$.
	Hence its \'{e}tale fundamental group is trivial, and Observation~\ref{lem:geometric-surjectivity} gives $\EtFundGrp{U_A}=1$.
	Therefore, we may assume that $\CardinalityOfSet{A}\geq 2$.

	Since every nontrivial finite solvable group has a nontrivial finite abelian quotient, it is enough to show that $\EtFundGrp{U_A}$ has no nontrivial finite abelian quotient.
	Let $G$ be a finite abelian quotient of $\EtFundGrp{U_A}$.
	By Observation~\ref{lem:geometric-surjectivity}, the composite of the natural homomorphism $\EtFundGrp{U_{A,\FieldAlgeClosure{\mathbb{Q}}}}\to\EtFundGrp{U_A}$ with the quotient map to $G$ is surjective.
	This composite extends to $\EtFundGrp{U_{A,\mathbb{Q}}}$.
	Since $G$ is abelian, the restriction of this extension to $\EtFundGrp{U_{A,\FieldAlgeClosure{\mathbb{Q}}}}$ is invariant under conjugation by $\EtFundGrp{U_{A,\mathbb{Q}}}$.
	It therefore factors through the $\AbsGalGrp{\mathbb{Q}}$-coinvariants of $\Abelianization{\EtFundGrp{U_{A,\FieldAlgeClosure{\mathbb{Q}}}}}$.
	Thus, there is a surjection
	\begin{equation*}
		\left(
		\Abelianization{\EtFundGrp{U_{A,\FieldAlgeClosure{\mathbb{Q}}}}}
		\right)_{\AbsGalGrp{\mathbb{Q}}}
		\twoheadrightarrow G.
	\end{equation*}
	Since all points of $A$ are $\mathbb{Q}$-rational, we have $\Abelianization{\EtFundGrp{U_{A,\FieldAlgeClosure{\mathbb{Q}}}}}\cong\widehat{\mathbb{Z}}(1)^{\CardinalityOfSet{A}-1}$ as an $\AbsGalGrp{\mathbb{Q}}$-module (see~\cite[The exact sequence (2.5)]{MR1040998}).
	Complex conjugation acts on this module as multiplication by~$-1$.
	Consequently, its coinvariants, and hence $G$, are annihilated by~$2$.
	Thus $G$ is an elementary abelian $2$-group.
	By Lemma~\ref{lem:good-prime-inertia-UA}, the order of $I_a(G)$ is prime to~$2$ for every $a\in A$.
	It follows that $I_a(G)=1$ for every $a\in A$.
	These subgroups generate $G$ by Observation~\ref{lem:geometric-surjectivity}; therefore, $G=1$.
	This proves the forward implication.

	Conversely, if $D_A$ is not normal crossings at the prime~$2$, then Proposition~\ref{prop:only-if} gives a surjective homomorphism $\EtFundGrp{U_A}\twoheadrightarrow\mathbb{Z}/2\mathbb{Z}$.
	Hence $\EtFundGrp{U_A}^{\mathrm{solv}}\neq 1$.
\end{proof}
\begin{corollary}
	\label{cor:easy-range}
	Conjecture~\ref{Iconjecture_NCat2} holds for $D_A$ if $\CardinalityOfSet{A}\notin\{3,4,5,6\}$.
\end{corollary}

\begin{proof}
	By Observation~\ref{lem:geometric-surjectivity}, there is a surjection $\EtFundGrp{U_{A,\FieldAlgeClosure{\mathbb{Q}}}}\twoheadrightarrow \EtFundGrp{U_A}$.
	The statement is clear when $\CardinalityOfSet{A}<2$.
	When $\CardinalityOfSet{A}=2$, the group $\EtFundGrp{U_{A,\FieldAlgeClosure{\mathbb{Q}}}}$ is isomorphic to $\widehat{\mathbb{Z}}$, and hence $\EtFundGrp{U_A}$ is abelian.
	Since every nontrivial profinite group has a nontrivial finite quotient, Theorem~\ref{thm:solvable-quotient} shows that $\EtFundGrp{U_A}$ is trivial when $D_A$ is normal crossings at the prime~$2$.

	Next, assume that $\CardinalityOfSet{A}\geq 7$.
	Every rational section specializes to one of the three $\mathbb{F}_{2}$-rational points of $\mathbb{P}^{1}_{\mathbb{F}_{2}}$.
	If $D_A$ were normal crossings at the prime~$2$, at most two components could pass through each of these points, which would imply $\CardinalityOfSet{A}\leq 6$.
	Therefore $D_A$ is not normal crossings at the prime~$2$.
	Thus, Conjecture~\ref{Iconjecture_NCat2} also holds in this range.
\end{proof}

\begin{example}
	\label{ex:p-square}
	Let $p$ be an odd prime, and put $A=\{0,p^{2},\infty\}$.
	By Lemma~\ref{lem:nc-at-prime-criterion}, the divisor $D_A$ is not normal crossings at~$p$ but is normal crossings at every other prime.
	In this case, $U_A$ has trivial \'etale fundamental group.

	Indeed, we first show that the inertia group attached to $D_\infty$ is trivial.
	We have $U_A=U_{\{0,p^2\}}\setminus D_\infty$, and $D_\infty\cong\Spec(\mathbb{Z})$ is a normal crossings divisor on $U_{\{0,p^2\}}$.
	Since $D_\infty$ has a closed point of every residue characteristic, Corollary~\ref{cor:tame-inertia-pro-complement} shows that its inertia group in $\EtFundGrp{U_A}$ is trivial.

	Let $G$ be a finite quotient of $\EtFundGrp{U_A}$.
	By Observation~\ref{lem:geometric-surjectivity}, the group $G$ is generated by the images of boundary inertia elements $\gamma_0,\gamma_{p^2},\gamma_\infty$ satisfying $\gamma_0\gamma_{p^2}\gamma_\infty=1$.
	Since $\gamma_\infty=1$, the group $G$ is cyclic.
	As $D_A$ is normal crossings at~$2$, Theorem~\ref{thm:solvable-quotient} gives $G=1$.
	Hence $\EtFundGrp{U_A}=1$.
\end{example}

\section{Finite non-abelian simple quotients}\label{sec:simple-quotients}

In this section, we study finite non-solvable quotients of $\EtFundGrp{U_A}$.
We keep the notation of Section~\ref{sec:triviality}.
Lemma~\ref{lem:good-prime-inertia-UA} immediately implies the following theorem, which is the basis for the results of this section:

\begin{theorem}
	\label{thm:general-nonsolvable-obstruction}
	Let $G$ be a nontrivial finite group, and write $o(G)$ for the set of prime divisors of $\CardinalityOfSet{G}$.
	If $D_A$ is normal crossings at every prime in $o(G)$, then there is no surjective homomorphism
	\begin{equation*}
		\EtFundGrp{U_A}\twoheadrightarrow G.
	\end{equation*}
\end{theorem}

\begin{proof}
	Assume that there exists a surjective homomorphism $\EtFundGrp{U_A}\twoheadrightarrow G$.
	For each $a\in A$, every prime divisor of $\CardinalityOfSet{I_a(G)}$ also lies in $o(G)$.
	On the other hand, Lemma~\ref{lem:good-prime-inertia-UA} shows that every prime divisor of $\CardinalityOfSet{I_a(G)}$ lies in the set of primes at which $D_A$ is not normal crossings.
	Therefore, we have $I_a(G)=1$ for every $a$.
	Observation~\ref{lem:geometric-surjectivity} implies that $G=1$, a contradiction.
\end{proof}

If $D_A$ is normal crossings at the prime~$2$ and $\EtFundGrp{U_A}$ is nontrivial, then a nontrivial finite quotient $G$ of minimal order is non-abelian simple by Theorem~\ref{thm:solvable-quotient}.
By the Feit--Thompson odd-order theorem, such a group has even order, and hence $2\in o(G)$.
Theorem~\ref{thm:general-nonsolvable-obstruction} would exclude $G$ if $D_A$ were normal crossings at every prime in $o(G)$, whereas in Conjecture~\ref{Iconjecture_NCat2} we assume this only at the prime~$2$.
Thus, the theorem cannot be applied directly.
We instead consider certain non-abelian finite simple groups $\PSL_{2}(q)$ with $q\geq 4$.
The key input is the $\mathbb{Q}$-rationality of boundary inertia generators.

\begin{definition}[{\cite[Definition~7.1.1]{MR2363329}}]
	\label{Qrationality}
	Let $G$ be a finite group and let $g\in G$ be an element of order $n$.
	We say that $g$ is $\mathbb{Q}$-rational if $g$ is conjugate to $g^{a}$ for every integer $a$ that is coprime to $n$.
\end{definition}

\begin{proposition}
	\label{prop:boundary-generators-rational}
	Let $G$ be a finite quotient of $\EtFundGrp{U_A}$ and let $g$ be a generator of $I_a(G)$ for some $a\in A$.
	Then $g$ is $\mathbb{Q}$-rational.
\end{proposition}

\begin{proof}
	The assertion follows from the Frobenius action on tame inertia and Dirichlet's theorem; see~\cite[Corollary~7.1.3]{MR2363329} for details.
\end{proof}

\begin{lemma}
	\label{lem:psl2-rational-odd}
	Let $q\geq 4$ satisfy one of the following conditions:
	\begin{enumerate}[(a)]
		\item $q=2^{f}$ for some $f\geq 2$;
		\item $q=3^{f}$ for some $f\geq 2$;
		\item $q$ is an odd prime power such that $2$ is not a square in $\mathbb{F}_{q}$.
	\end{enumerate}
	Then every nontrivial $\mathbb{Q}$-rational element of odd order in $\PSL_{2}(q)$ has order~$3$.
\end{lemma}

\begin{proof}
	We use the standard description of semisimple and unipotent conjugacy classes in $\PSL_{2}(q)$; see~\cite[Sections~2.1 and~26.2]{MR2850737}.
	Assume first that $g$ is semisimple.
	Since $g$ has odd order and is $\mathbb{Q}$-rational, it is conjugate to $g^2$.
	Choose the unique lift $\widetilde{g}\in\SL_{2}(q)$ of odd order, and let its eigenvalues be $\lambda$ and $\lambda^{-1}$.
	Since $g$ and $g^2$ are conjugate in $\PSL_{2}(q)$, their unique odd-order lifts $\widetilde{g}$ and $\widetilde{g}^2$ are conjugate in $\SL_{2}(q)$.
	Hence
	\begin{equation*}
		\{\lambda,\lambda^{-1}\}
		=
		\{\lambda^2,\lambda^{-2}\}.
	\end{equation*}
	Thus, either $\lambda=1$ or $\lambda^3=1$.
	Therefore $g=1$ or $g^3=1$, and a nontrivial $g$ has order~$3$.

	It remains to consider the case where $g$ is unipotent.
	Write $q=p^f$.
	Since $g$ is nontrivial and unipotent, it has order~$p$.
	Since $g$ has odd order, we have $p\neq 2$.
	If $p=3$, the assertion follows immediately.
	We may therefore assume that condition~\textup{(c)} holds.
	The element $g$ is represented by
	\begin{equation*}
		u(a)=
		\begin{pmatrix}
			1 & a \\
			0 & 1
		\end{pmatrix}
		\qquad (a\in \mathbb{F}_{q}^{\times}).
	\end{equation*}
	The $\mathbb{Q}$-rationality of $g$ implies that $g$ and $g^2$ are conjugate.
	Since $u(a)^2=u(2a)$, the images of $u(a)$ and $u(2a)$ in $\PSL_{2}(q)$ are conjugate.
	For odd $q$, this is possible only if $2$ is a square in $\mathbb{F}_{q}^{\times}$, contradicting condition~\textup{(c)}.
	Thus $g$ cannot be unipotent.
	This completes the proof.
\end{proof}

\begin{remark}
	There are isomorphisms
	\begin{equation*}
		\PSL_{2}(4)\cong A_{5},\qquad \PSL_{2}(9)\cong A_{6}.
	\end{equation*}
	Hence Lemma~\ref{lem:psl2-rational-odd} also implies that \emph{every nontrivial odd-order $\mathbb{Q}$-rational element of $A_{5}$ or $A_{6}$ has order~$3$}.
	On the other hand, for alternating groups $A_{n}$ with $n\geq 7$, the situation is different: odd-order $\mathbb{Q}$-rational elements need not have order $3$.
\end{remark}

\begin{theorem}
	\label{thm:small-simple-obstruction}
	Assume that $D_A$ is normal crossings at the prime~$2$.
	Assume moreover that either $\CardinalityOfSet{A}=3$ or $D_A$ is normal crossings at the prime~$3$.
	Let $q\geq 4$ satisfy one of the following conditions:
	\begin{enumerate}[(a)]
		\item $q=2^{f}$ for some $f\geq 2$;
		\item $q=3^{f}$ for some $f\geq 2$;
		\item $q$ is an odd prime power such that $2$ is not a square in $\mathbb{F}_{q}$.
	\end{enumerate}
	Then there is no surjective homomorphism
	\begin{equation*}
		\EtFundGrp{U_A}\twoheadrightarrow \PSL_{2}(q).
	\end{equation*}
\end{theorem}

\begin{proof}
	Assume that there exists a surjective homomorphism $\EtFundGrp{U_A}\twoheadrightarrow G \coloneq \PSL_{2}(q)$.
	If $A=\varnothing$, then $U_A=\mathbb{P}^{1}_{\mathbb{Z}}$ and $\EtFundGrp{\mathbb{P}^{1}_{\mathbb{Z}}}\cong \EtFundGrp{\Spec(\mathbb{Z})}=1$, contradicting the assumed surjection.
	Hence we may assume that $A\neq\varnothing$.
	Let $a\in A$ and choose a generator $g$ of $I_a(G)$.
	By Lemma~\ref{lem:good-prime-inertia-UA} and Proposition~\ref{prop:boundary-generators-rational}, the element $g$ has odd order and is $\mathbb{Q}$-rational.
	Hence Lemma~\ref{lem:psl2-rational-odd} implies that $g$ is trivial or has order~$3$.

	When $D_A$ is normal crossings at the prime~$3$, Lemma~\ref{lem:good-prime-inertia-UA} excludes the possibility that $g$ has order~$3$.
	Since $a$ was arbitrary, this gives $I_a(G)=1$ for every $a\in A$, which contradicts Observation~\ref{lem:geometric-surjectivity}.

	It remains to consider the case $\CardinalityOfSet{A}=3$.
	Write $A=\{a_1,a_2,a_3\}$.
	By Observation~\ref{lem:geometric-surjectivity}, the quotient $\EtFundGrp{U_A}\twoheadrightarrow G$ is generated by the images of three boundary inertia elements $\gamma_1,\gamma_2,\gamma_3$ satisfying $\gamma_1\gamma_2\gamma_3=1$.
	Since the images of the $\gamma_i$ in $G$ have order dividing $3$, the quotient factors through the profinite completion of
	\begin{equation*}
		\Delta(3,3,3)
		\coloneq
		\langle x_1,x_2,x_3\mid x_1^3=x_2^3=x_3^3=x_1x_2x_3=1\rangle.
	\end{equation*}
	The triangle group $\Delta(3,3,3)$ fits into an exact sequence
	\begin{equation*}
		1\to \mathbb{Z}^{2}\to \Delta(3,3,3)\to \mathbb{Z}/3\mathbb{Z}\to 1.
	\end{equation*}
	In particular, its profinite completion is prosolvable.
	This contradicts the fact that $\PSL_{2}(q)$ is non-abelian simple, hence non-solvable, for $q\geq 4$ (see~\cite[Theorem~24.17 and Example~24.22]{MR2850737}).
	This completes the proof.
\end{proof}

\begin{remark}
	Let $q=p^{f}$ be an odd prime power.
	The element $2$ is not a square in $\mathbb{F}_{q}$ if and only if $f$ is odd and $p\equiv \pm 3 \pmod{8}$.
	Indeed, an element $a\in\mathbb{F}_{p}^{\times}$ is a square in $\mathbb{F}_{q}^{\times}$ if and only if $a^{(q-1)/2}=1$.
	We have
	\begin{equation*}
		a^{(q-1)/2}
		=
		\left(a^{(p-1)/2}\right)^{(q-1)/(p-1)}.
	\end{equation*}
	Since
	\begin{equation*}
		\frac{q-1}{p-1}=1+p+\cdots+p^{f-1}\equiv f\pmod{2},
	\end{equation*}
	if $f$ is even, then $(q-1)/(p-1)$ is even, and hence $a^{(q-1)/2}=1$ for every $a\in\mathbb{F}_{p}^{\times}$.
	Thus, every element of $\mathbb{F}_{p}^{\times}$, in particular $2$, is a square in $\mathbb{F}_{q}$.
	Next, assume that $f$ is odd.
	In this case, $(q-1)/(p-1)$ is odd, and therefore
	\begin{equation*}
		a^{(q-1)/2}=a^{(p-1)/2}.
	\end{equation*}
	Consequently, $2$ is a square in $\mathbb{F}_{q}$ if and only if $2$ is a square in $\mathbb{F}_{p}$.
	By the supplementary law for quadratic reciprocity, we have
	\begin{equation*}
		\left(\frac{2}{p}\right)
		=
		(-1)^{(p^{2}-1)/8}.
	\end{equation*}
	Hence $2$ is not a square in $\mathbb{F}_{p}$ if and only if $p\equiv \pm 3 \pmod{8}$.
\end{remark}

\begin{remark}
	If Conjecture~\ref{Iconjecture_NCat2} is true, then the triviality of $\EtFundGrp{U_A}$ is equivalent to the condition that $D_A$ is normal crossings at the prime~$2$.
	On the other hand, when $D_A$ is not normal crossings at the prime~$2$, it is not known whether the \'etale fundamental group is finite.
	This leads to the following question:

	\medskip

	\begin{quote}
		Is $\EtFundGrp{U_A}$ infinite whenever $D_A$ is not normal crossings at the prime~$2$?
	\end{quote}

	\medskip

	\noindent
	This question is related to finiteness questions for tame fundamental groups of arithmetic schemes; see~\cite{MR1883386} and~\cite{MR2578565}.
	For a related arithmetic analogue of the Lefschetz and Nori theorems for fundamental groups, see Bost--Charles~\cite{bost2022quasiprojectiveformalanalyticarithmeticsurfaces}.
\end{remark}

\section*{Acknowledgements}
The first author is sincerely grateful to Prof. Akio~Tamagawa for his generous guidance and helpful suggestions throughout this work.
This work was supported by the Japan Society for the Promotion of Science (JSPS) KAKENHI Grant Numbers 25KJ0125 and 23KJ0881.

\begingroup
\raggedright
\printbibliography
\endgroup

\end{document}